\documentclass{amsart}
\usepackage{amssymb}
\usepackage{hyperref}
\usepackage{color}

\numberwithin{equation}{subsection}

\newtheorem{theorem}{Theorem}[section]
\newtheorem*{theorem*}{Theorem}
\newtheorem{corollary}[theorem]{Corollary}

\newtheorem{lemma}[theorem]{Lemma}

\theoremstyle{definition}
\newtheorem{defn}[theorem]{Definition}

\def\alt#1{\mathrm{A}_{#1}}
\def\sym#1{\mathrm{S}_{#1}}

\thanks{The research of the first author is supported by the Australian Research Council grant
DP120100446. This work was done whilst the second author was
visiting The University of Western Australia as a Cheryl E. Praeger
Visiting Research Fellow. The authors would like to thank the
anonymous referee for their extremely helpful remarks.}

\begin{document}

\title[Generating alternating or symmetric groups]{A note on the probability of generating alternating or symmetric groups}

\author[Luke Morgan]{Luke Morgan}
\address{Luke Morgan,
Centre for the Mathematics of Symmetry and Computation, School of Mathematics and Statistics \\
University of Western Australia \\
35 Stirling Highway \\
Crawley, WA
6009, Australia}
\email{luke.morgan@uwa.edu.au}

\author[Colva M.~ Roney-Dougal]{Colva M. Roney-Dougal}
\address{Colva M. Roney-Dougal,
School of Mathematics and Statistics\\
  University of St Andrews \\
North Haugh \\
St Andrews \\
Fife KY16 9SS,
U.~K.}
\email{colva.roney-dougal@st-andrews.ac.uk}


\begin{abstract}
We improve on recent estimates for the probability of generating the
alternating and symmetric groups $\alt{n}$ and $\sym{n}$. In
particular we find the sharp lower bound, if the probability is given
by a quadratic in $n^{-1}$. This leads to improved bounds on the largest number
$h(\alt{n})$ such that a direct product of $h(\alt{n})$
 copies of $\alt{n}$ can be generated by two elements. 
\end{abstract}


\maketitle

\section{Introduction}

For a group $X = \sym{n}$ or $\alt{n}$, we write $p(X)$ for the probability that two elements of $X$ generate a group that contains $\alt{n}$. 
In \cite{Dixon69}, Dixon proved that $p(\sym{n}) \rightarrow 1$ as $n \rightarrow \infty$. In \cite{Dixon05} he sharpened this statement to 
$$p(\sym{n}) = 1 - \frac{1}{n} - \frac{1}{n^2} - \frac{4}{n^3} - \frac{23}{n^4} - \frac{171}{n^5} - \frac{1542}{n^6} + O(n^{-7}).$$
For many applications, numerical results are needed, rather than
asymptotics. In \cite{marotitamburini} Mar{\'o}ti and Tamburini proved
explicit upper and lower bounds $$1 - \frac{1}{n} - \frac{13}{n^2} <
p(X) \leqslant 1 - \frac{1}{n} + \frac{2}{3n^2}.$$
In this present note, we find the best possible lower bound of this
 type, and a close-to-optimal upper bound. 

\begin{theorem}
\label{main}
Let $X=\alt{n}$ or $X=\sym{n}$ with $n \geqslant 5$. Then
$$ 1-  \frac{1}{n} - \frac{8.8}{n^2}  \leqslant   p(X) < 1-\frac{1}{n} - \frac{0.93}{n^2}.$$ Equality holds in the lower bound if and only if $n = 6$. 
\end{theorem}

In fact, for $n \geqslant 14$, we prove that $1 - \frac{1}{n}
-\frac{7.5}{n^2} < p(X) <  1 - \frac{1}{n} - \frac{0.93}{n^2}$. The result for smaller $n$ comes from the values for $p(X)$  in Table~\ref{tab:numbers} (taken from \cite[Table 4.1]{ninathesis}).

Hall \cite{Hall} considered the largest number $h(S)$
such that a direct product of $h(S)$ copies of a non-abelian finite simple group $S$  
can be generated by two elements,
 and proved that   $h(S) =
p(S)|S|/|\mathrm{Out}(S)|$. 
The function $h(S)$ has received 
considerable attention recently; we refer the reader to \cite{marotitamburini} for more
discussion and references  and to \cite{nina} for lower bounds on
$h(S)$ for all non-abelian finite simple groups $S$.  The new bounds above yield:

\begin{corollary}
Let $n$ be an integer with $n\geqslant 14$. Then
$$ \left ( 1 - \frac{1}{n} - \frac{7.5}{n^2} \right ) \left
  (\frac{n!}{4} \right ) < h(\alt{n}) < \left ( 1 - \frac{1}{n}  - \frac{0.93}{n^2}\right ) \left ( \frac{n!}{4} \right ).$$
\end{corollary}


Let $m(S)$ denote the minimal index of a proper subgroup of a group $S$.
In \cite{liebeckshalev}, it is proved that there exist absolute
constants $c_1$ and $c_2$ such that $1 - c_1/m(S) < p(S) < 1  -
c_2/m(S)$, for all non-abelian finite simple groups $S$. 
For $i=1,2$, let $a_i$ denote the value of $c_i$ for the family of simple alternating groups.

\begin{corollary}
For $n \geqslant 5$, 
$$1 - \frac{2.468}{n} < p(\alt{n}) < 1 - \frac{1}{n},$$ and hence $a_1 \leqslant 2.468$ and $a_2 \geqslant 1$. 
\end{corollary}

\section{Proof of Theorem~\ref{main}}

\begin{defn}
For $X=\alt{n}$ or $\sym{n}$ we let
$p_{{\mbox{\footnotesize{intrans}}}}(X)$ and
$p_{{\mbox{\footnotesize{trans}}}}(X)$  be the probability that two
elements chosen randomly from $X$ generate a subgroup of an intransitive maximal subgroup of
$X$, or a subgroup of a transitive maximal subgroup of $X$
 other than $\alt{n}$, respectively.
\end{defn}


\begin{lemma}
\label{lower}
Let $X=\alt{n}$ or $\sym{n}$ with $n \geqslant 14$. Then
$$p_{\footnotesize{\mbox{intrans}}} (X)  < \frac{1}{n} +\frac{2.7}{n^2}. $$
\end{lemma}
\begin{proof}
We prove the result for $\sym{n}$, the arguments for $\alt{n}$ are identical. Let $x,y \in \sym{n}$ and suppose 
that $Y:=\langle x, y \rangle$ is contained in an intransitive maximal subgroup. Then $Y$ is contained in a subgroup conjugate to $\sym{k} \times \sym{n-k}$ for some $1 \leqslant k \leqslant  \lfloor \frac{n-1}{2} \rfloor$.

Let $k \in \{1, \ldots, n-1\}$. Then 
the probability that $Y \leqslant \sym{k} \times \sym{n-k}$ is bounded by
$$\binom{n}{k}\left( \frac{k!(n-k)!}{n!}\right)^2 = \binom{n}{k}^{-1}.$$
So the probability that $Y \leqslant \sym{1} \times \sym{n-1}$ is at most $\frac{1}{n}$, and 
the probability that $Y \leqslant \sym{2} \times \sym{n-2}$ and $Y$  is transitive on the orbit of size $2$ is bounded by 
$$\frac{3}{4} \frac{2}{n(n-1)}= \frac{3}{2n(n-1)}.$$ 
Similarly, the probability that $Y \leqslant \sym{3} \times \sym{n-3}$ and
$Y$ is transitive on the orbit of length 3 is
$$\frac{13}{18} \binom{n}{3}^{-1} = \frac{13}{3n(n-1)(n-2)}.$$
Now the probability that $Y \leqslant \sym{k} \times \sym{n-k}$ for some $4 \leqslant k \leqslant \lfloor \frac{n-1}{2} \rfloor$ is
$$\sum_{k=4}^{\lfloor \frac{n-1}{2} \rfloor} \frac{1}{\binom{n}{k}} 
\leqslant  \sum_{k=4}^{\lfloor \frac{n-1}{2} \rfloor} \frac{1}{\binom{n}{4}} 
\leqslant \frac{12(n-7)}{n(n-1)(n-2)(n-3)}.$$
We now observe that, since $n \geqslant 14$, 
$$\frac{3}{2n(n-1)} +\frac{13}{3n(n-1)(n-2)} + \frac{12(n-7)}{n(n-1)(n-2)(n-3)} < \frac{2.7}{n^2}$$
which completes the proof.
\end{proof}

\begin{lemma}
\label{upper}
Let $X = \alt{n}$ or $\sym{n}$, with $n \geqslant 14$. Then 
$$p_{\mbox{\footnotesize{intrans}}}(X) > \frac{1}{n} + \frac{0.93}{n^2} .$$
\end{lemma}
\begin{proof}
We observe that $p_\text{intrans}(X)$ is bounded below by the probability that a random  pair of elements of $X$ generate a subgroup with a fixed point, or with an orbit of size $2$. For $X = \sym{n}$, we bound $p_\text{intrans}(X)$ by doing inclusion-exclusion to depth $2$ on the union of the sets $(\sym{n})_
\alpha$, with $1 \leqslant \alpha \leqslant n$, and $(\sym{n})_{\{\alpha, \beta\}} \setminus (\sym{n})_{(\alpha, \beta)}$, with $1 \leqslant \alpha < \beta \leqslant n$. 
We find that $p_\text{intrans}(X)$ is greater than
$$
\begin{array}{c}
\frac{1}{n}   
+ \frac{3}{4} \frac{2(n-2)!}{n!}
- \frac{(n-2)!}{2n!} 
- \frac{3}{4} \binom{n}{1}\binom{n-1}{2}\left(\frac{2(n-3)!}{n!}\right)^2 
- \left( \frac{3}{4} \right)^2 \frac{\binom{n}{2}\binom{n-2}{2}}{2}\left(\frac{4(n-4)!}{n!}\right)^2 
\end{array}
$$
Thus
$$p_\text{intrans}(X) \geqslant  \frac{1}{n}  + \frac{8n^2 - 52 n + 75}{8n(n-1)(n-2)(n-3)}$$
which, since $n\geqslant 14$, is  greater than $\frac{1}{n} + \frac{0.93}{n^2}$.
\end{proof}


\begin{proof}[Proof of Theorem~\ref{main}]
For the upper bound we use Lemma~\ref{upper}. For the lower bound, note that 
$$1 - p(X) =  p_\text{intrans}(X) + p_\text{trans}(X).$$ It follows
from the proofs of
\cite[Lemmas 3.1 and 4.3]{marotitamburini}
that $p_\text{trans}(X) \leqslant \frac{4.8}{n^2}$. Combining this with Lemma~\ref{lower} gives  the theorem.
\end{proof}

In Table~\ref{tab:numbers} we record the value of $p(\alt{n})$ and $p(\sym{n})$
for $n \leqslant 13$, together with our lower and upper bounds as
stated in Theorem~\ref{main}. All values are correct to three decimal places.

\begin{table}\caption{Precise values and bounds for $p(X)$
} \label{tab:numbers}
{\small{$$\begin{array}{c||c|c|c|c|c|c|c|c|c}
n & 5 & 6 & 7 & 8 & 9 & 10 & 11 & 12 & 13 \\
\hline
p(\alt{n}) = & 0.633 & 0.588 & 0.726 & 0.739 & 0.848 & 0.875 & 0.893 & 0.902 & 0.913 \\
p(\mathrm S_n) =  & 0.633 & 0.588 & 0.795 & 0.796 & 0.859 & 0.875 & 0.894 & 0.903 & 0.913 \\
p(X) \geqslant & 0.448 & 0.588 & 0.677 & 0.737 & 0.780 & 0.812 & 0.836 & 0.855 & 0.871\\
p(X) \leqslant & 0.763 & 0.808 & 0.839 & 0.861 &  0.878 & 0.891 & 0.902 & 0.911 & 0.918 \\
\end{array}
$$}}
\end{table}

\end{document}